\documentclass[a4paper,reqno,12pt]{amsart} 
\usepackage{amsmath} 

\usepackage{amssymb}      

\usepackage{yhmath}
\usepackage{mathdots}
\usepackage{MnSymbol}

\usepackage{amsthm}      

\usepackage{tikz,amsmath}
\usetikzlibrary{arrows}

\usepackage{epsfig}      

\usepackage{graphicx}
\usepackage[normalem]{ulem}
\usepackage[all,line,
]{xy} 
\CompileMatrices 

   \newcommand{\A}{{\mathbb A}}
   \newcommand{\fg}{{finitely generated }}
   \newcommand{\Aff}{{\operatorname{Aff}}}
 \newcommand{\dist}{{\operatorname{dist}}}
   \newcommand{\Spec}{\operatorname{Spec}}
   \newcommand{\gr}{\operatorname{gr}}
   \newcommand{\Hom}{\operatorname{Hom}} 
\newcommand{\Diag}{\operatorname{Diag}}  
\newcommand{\Pext}{\operatorname{Pext}}
\newcommand{\Ad}{\operatorname{Ad}}
   \newcommand{\CH}{\operatorname{CH}} 
\newcommand{\Log}{\operatorname{Log}} 
\newcommand{\id}{\operatorname{id}} 
\newcommand{\Aut}{\operatorname{Aut}}
\newcommand{\eq}[1]{\begin{equation}#1\end{equation}}
\newcommand{\eQ}[1]{\begin{equation*}#1\end{equation*}}
\newcommand{\spl}[1]{\begin{split}#1\end{split}}
 \newcommand{\Ext}{\operatorname{Ext}} 
\newcommand{\diag}{\operatorname{diag}}
\newcommand{\cpah}{completely positive asymptotic homomorphism}
\newcommand{\dcpah}{discrete completely positive asymptotic homomorphism}
\newcommand{\ah}{asymptotic homomorphism }
\newcommand{\cone}{\operatorname{cone}}
\newcommand{\Bott}{\operatorname{Bott}}
\newcommand{\im}{\operatorname{im}}
\newcommand{\co}{\operatorname{co}}
\newcommand{\Span}{\operatorname{Span}}
\newcommand{\Tr}{\operatorname{Tr}}
\newcommand{\diam}{\operatorname{diam}}
\newcommand{\Int}{\operatorname{Int}}
\newcommand{\Per}{\operatorname{Per}}
 \newcommand{\Depth}{\operatorname{Depth}} 
 \newcommand{\period}{\operatorname{period}} 
 \newcommand{\mit}{\operatorname{ MIN\left(\complement L\right)   }}
 \newcommand{\ins}{\operatorname{int}} 
 \newcommand{\Rec}{\operatorname{Rec}}
 \newcommand{\supp}{\operatorname{supp}}
\newcommand{\con}{\operatorname{context}}
\newcommand{\gd}{\operatorname{gcd}}
\newcommand{\R}{\operatorname{\mathcal R^-}}
\newcommand{\Det}{\operatorname{Det}}
\newcommand{\Dim}{\operatorname{Dim}}
\newcommand{\Rank}{\operatorname{Rank}}
\newcommand{\Is}{\operatorname{Is}}
\newcommand{\alg}{\operatorname{alg}}
 \newcommand{\Prim}{\operatorname{Prim}}
 \newcommand{\tr}{\operatorname{tr}}
\newcommand{\inv}{^{-1}}
\newcommand{\N}{{\mathbb{N}}}
\newcommand{\Z}{{\mathbb{Z}}}
\newcommand{\T}{{\mathbb{T}}}
\newcommand{\cip}{{\mathcal{M}}}
\newcommand{\cipp}{{\cip_{\per}}}
\newcommand{\cipa}{{\cip_{\aper}}}
\newcommand{\ev}{\operatorname{ev}}
\newcommand{\Sec}{\operatorname{Sec}}
\newcommand{\sign}{\operatorname{sign}}
\newcommand{\pt}{\operatorname{point}}
\newcommand{\cof}{\operatorname{cof}}
\newcommand{\Cone}{\operatorname{Cone}}
\newcommand{\val}{\operatorname{val}}
\newcommand{\INT}{\operatorname{int}}
\newcommand{\Ro}{\operatorname{RO}}
\newcommand{\Orb}{\operatorname{Orb}}
\newcommand{\Sp}{\operatorname{Sp}}
\newcommand{\Cay}{\operatorname{Cay}}
\newcommand{\Geo}{\operatorname{Geo}}
\newcommand{\words}[1][]{\mathbb{W}_{#1}(S)}
\newcommand{\Br}{\operatorname{Br}}
\newcommand{\Wan}{\operatorname{Wan}}
\newcommand{\Ray}{\operatorname{Ray}}
\newcommand{\cl}{\operatorname{Cl}}
\newcommand{\Nar}{\operatorname{Nar}}

   \newcommand{\nsum}{{\sum_{i=1}^n}}
   \newcommand{\ncap}{{\cap_{i=1}^n}}
   \newcommand{\bncap}{{\bigcap_{i=1}^n}}
   \newcommand{\bnoplus}{{\bigoplus_{i=1}^n}}
   \newcommand{\noplus}{{\oplus_{i=1}^n}}
   \newcommand{\nprod}{{\prod_{i=1}^n}}
   \newcommand{\ncup}{{\cup_{i=1}^n}}
   \newcommand{\bncup}{{\bigcup_{i=1}^n}}
\newcommand{\defeq}{\stackrel{\text{def}}{=}}
\DeclareMathOperator{\orb}{Orb}
\DeclareMathOperator{\per}{Per}
\DeclareMathOperator{\aper}{Aper}

   \theoremstyle{plain}
   \newtheorem{thm}{Theorem}[section]
   \newtheorem{prop}[thm]{Proposition}
   \newtheorem{lemma}[thm]{Lemma}  
   \newtheorem{cor}[thm]{Corollary}
   \theoremstyle{definition}

   \theoremstyle{remark}
   
   \newtheorem{remark}[thm]{Remark}

\newtheorem{q}[thm]{Question}

\newtheorem{poem}[thm]{Additional properties}

   \newcommand{\refthm}[1]{Theorem~\ref{#1}}
   \newcommand{\refprop}[1]{Proposition~\ref{#1}}
   \newcommand{\reflemma}[1]{Lemma~\ref{#1}}
   \newcommand{\refcor}[1]{Corollary~\ref{#1}}
   \newcommand{\refdefn}[1]{Definition~\ref{#1}}
   \newcommand{\refremark}[1]{Remark~\ref{#1}}
   \newcommand{\refexample}[1]{Example~\ref{#1}}

\providecommand{\abs}[1]{\lvert#1\rvert}
\providecommand{\norm}[1]{\lVert#1\rVert}

\usepackage{mdframed,xcolor}

\definecolor{mybgcolor}{gray}{0.8}
\definecolor{myframecolor}{rgb}{.647,.129,.149}

\mdfdefinestyle{mystyle}{
  usetwoside=false,
  skipabove=0.6em plus 0.8em minus 0.2em,
  skipbelow=0.6em plus 0.8em minus 0.2em,
  innerleftmargin=.25em,
  innerrightmargin=0.25em,
  innertopmargin=0.25em,
  innerbottommargin=0.25em,
  leftmargin=-.75em,
  rightmargin=-0em,
  topline=false,
  rightline=false,
  bottomline=false,
  leftline=false,
  backgroundcolor=mybgcolor,
  splittopskip=0.75em,
  splitbottomskip=0.25em,
  innerleftmargin=0.5em,
  leftline=true,
  linecolor=myframecolor,
  linewidth=0.25em,
}

\newmdenv[style=mystyle]{important}

\usepackage{lipsum}

   \numberwithin{equation}{section}

\title[On the possible temperatures for flows on an AF-algebra]{On the possible temperatures for flows on an AF-algebra}
\author{Klaus Thomsen}

        \email{matkt@imf.au.dk}
        \address{Institut for matematiske fag, Ny Munkegade, 8000 Aarhus C, Denmark}

\newcommand\xsquid[1]{\mathrel{\smash{\overset{#1}{\rightsquigarrow}}}}
\newcommand\soverline[1]{\smash{\overline{#1}}}

\DeclareMathOperator\coker{coker}

\begin{document}

\maketitle

\begin{abstract} It is shown that there is a unital simple mono-tracial AF algebra A with the property that for every compact set K of real numbers containing zero there is a flow on A for which the set of inverse temperatures is exactly K.
 \end{abstract}

\section{Introduction}

The work \cite{BEH} by Bratteli, Elliott and Herman demonstrated that the set of inverse temperatures for a one parameter group of automorphisms of a simple unital $C^*$-algebra can be an arbitrary closed set of real numbers. For general reasons the set must be closed, and their examples show that this is the only restriction. It follows also from \cite{BEH} that the simplexes of equilibrium states can vary wildly with the temperature. However, in the quantum statistical models where the theory is used the underlying $C^*$-algebra is often an AF algebra or even a UHF algebra, and it was not before the work of Kishimoto, in \cite{Ki2} and \cite{Ki3}, that one could begin to suspect that a similar wild behaviour can occur for flows on simple AF algebras. Indeed, in \cite{Ki2} Kishimoto constructed flows on simple unital AF algebras and UHF algebras with a non-trivial structure of KMS states and in \cite{Ki3} he opened up for the possibility of having flows on simple unital AF algebras for which the set of possible inverse temperatures is not the entire real line. Recall that by results of Powers and Sakai, \cite{PS}, any approximately inner flow on a unital $C^*$-algebra with a trace state must have $\beta$-KMS states for all $\beta$. In particular, any approximately inner flow on a unital AF algebra must have equilibrium states at every temperature. In addition, Powers and Sakai conjectured that any flow on a UHF algebra is approximately inner, which, if true, would imply that all flows on a UHF algebra have equilibrium states at every temperature. The AF version of the Powers-Sakai conjecture was refuted by Kishimoto in \cite{Ki3}, and the actual conjecture by Matui and Sato in \cite{MS}. The possibility of having other sets of inverse temperatures than $\mathbb R$ for flows on a simple unital AF algebra and perhaps even on a UHF algebra has been actualized by these results. It can be shown that for the examples of Kishimoto in \cite{Ki3} and Matui and Sato in \cite{MS}, the flows that are shown not to be approximately inner have no $\beta$-KMS states when $\beta \neq 0$. (See Remark \ref{11-11-20} below.) Hence, up to now, for all flows on AF algebras for which the set of possible inverse temperatures has been found, the set has either been $\mathbb R$ or $\{0\}$. The main purpose of the present paper is to show that the ideas of Matui and Sato can be combined with the approach of Bratteli, Elliott and Herman to get an example of a simple unital AF algebra $A$ such that for any compact set of real numbers containing zero there is a $2\pi$-periodic flow $\theta^K$ on $A$ with the set of inverse temperatures equal to $K$, and for any $\beta$ in $K$ the $\beta$-KMS state is unique. In fact, we will show that one of the AF algebras considered in \cite{BEH} is such an algebra. In addition we combine the results with the authors recent work on the variation with $\beta$ of the simplexes of $\beta$-KMS states for flows on UHF algebras. 

The fundamental idea behind the new examples is the same as in \cite{BEH}, \cite{Ki3} and \cite{MS}, namely to consider the dual action on well-chosen crossed products by the integers, and as in all these papers the success depends on results from the classification of simple $C^*$-algebras. Besides the classification of AF algebras, which was fundamental already in \cite{BEH}, we depend here on the recent work by Castillejos, Evington, Tikuisis, White and Winter in \cite{CETWW}.

\smallskip

\emph{Acknowledgement} The work was supported by the DFF-Research Project 2 `Automorphisms and Invariants of Operator Algebras', no. 7014-00145B.

\section{Statement of results}\label{statresults}

In this paper all $C^*$-algebras are assumed to be separable and all traces and weights on a $C^*$-algebra are required to be non-zero, densely defined and lower semi-continuous. Concerning weights and in particular KMS weights we shall use notation and results from Sections 1.1 and 1.3 in \cite{KV}. Let $A$ be a $C^*$-algebra and $\theta$ a flow on $A$. Let $\beta \in \mathbb R$. A $\beta$-KMS weight for $\theta$ is a weight $\omega$ on $A$ such that $\omega \circ \theta_t = \omega$ for all $t$, and 
\begin{equation}\label{27-10-20c}
\omega(a^*a) \ = \ \omega\left(\theta_{-\frac{i\beta}{2}}(a) \theta_{-\frac{i\beta}{2}}(a)^*\right) \ \  \ \forall a \in D(\theta_{-\frac{i\beta}{2}}) \ .
\end{equation}
In particular, a $0$-KMS weight for $\theta$ is a $\theta$-invariant trace. It was shown by Kustermans in Theorem 6.36 of \cite{Ku} that this definition agrees with the one introduced by Combes in \cite{C}.
It is because of the formulation given by \eqref{27-10-20c}, which was not available when \cite{BEH} was written, that we are able to work with KMS weights throughout the present work. A bounded $\beta$-KMS weight is called a $\beta$-KMS functional and a $\beta$-KMS state when it is of norm $1$.

The elements of the additive group $\mathbb Z\left[t,t^{-1},(1-t)^{-1}\right]$ of polynomials in $t,t^{-1}$ and $(1-t)^{-1}$ with integer coefficients will be considered as continuous functions on the open interval $]0,1[$. To simplify the notation we denote this group by $G_0$. Let $F$ be a closed non-empty subset of $[0,1]$ which neither contains $0$ nor $1$. Set
$$
G^+_F \ = \ \left\{ f \in  G_0: \ f(t) > 0 \ \forall t \in F \right\} \cup \{0\} \ .
$$
With $G_F^+$ as the semigroup of positive elements the group $G_0$ is a simple dimension group by Corollary 2.2 in \cite{BEH}. As in \cite{BEH} we denote by $A_F$ the unique AF algebra whose $K_0$-group with dimension scale is isomorphic to $(G_0,G^+_F)$, \cite{EHS}. The main results of the paper are the following.

\begin{thm}\label{08-11-20d} Let $F \subseteq [0,1]$ and $F_1 \subseteq [0,1]$ be closed subsets of $[0,1]$ such that
\begin{itemize}
\item $F \neq \emptyset$,
\item $(F \cup F_1) \cap \{0,1\} = \emptyset$, and
\item $1/2 \in (F\cup F_1)^c \cup (F \cap F_1) $.
\end{itemize} 
There is a $2\pi$-periodic flow $\theta$ on $A_F$ such that for $\beta \neq 0$ there is a $\beta$-KMS weight for $\theta$ if and only if $\frac{e^{-\beta}}{1+e^{- \beta}} \in F_1$. For each such $\beta$ the $\beta$-KMS weight is unique up to multiplication by scalars. All traces on $A_F$ are $\theta$-invariant and the cone of $0$-KMS weights for $\theta$ is affinely homeomorphic to the cone of bounded Borel measures on $F$. 
\end{thm} 
 
The AF algebra $A_F$ is simple and stable. The following is a unital version of Theorem \ref{08-11-20d}.

\begin{thm}\label{08-11-20e} Let $F$ and $F_1$ be as in Theorem \ref{08-11-20d} and let $p \in A_F$ be a projection which represents the constant function $1 \in G_0$. 
There is a $2\pi$-periodic flow $\theta$ on $pA_Fp$ such that for $\beta \neq 0$ there is a $\beta$-KMS state for $\theta$ if and only if $\frac{e^{-\beta}}{1+e^{- \beta}} \in F_1$. For each such $\beta$ the $\beta$-KMS state is unique. All traces on $pA_Fp$ are $\theta$-invariant and the set of $0$-KMS states for $\theta$ is affinely homeomorphic to the set of Borel probability measures on $F$. 
\end{thm} 

The two theorems are closely related, and we shall derive Theorem \ref{08-11-20d} from Theorem \ref{08-11-20e}. Consider a compact subset $K$ of real numbers and apply Theorem \ref{08-11-20e} with $F = \{1/2\}$ and 
$$
F_1 = \left\{ \frac{e^{- \beta}}{1+e^{-\beta}}: \ \beta \in K \right\} \ \cup \{1/2\} \ .
$$
This gives the following

\begin{cor}\label{08-11-20f} Let $p$ be a projection in $A_{\{1/2\}}$ representing the constant function $1 \in G_0$. Then $pA_{\{1/2\}}p$ is a simple unital mono-tracial AF algebra with the following property: For every compact set $K$ of real numbers there is a $2 \pi$-periodic flow $\theta^K$ on $pA_{\{1/2\}}p$ such that there is a $\beta$-KMS state for $\theta^K$ if and only if $\beta \in K \cup \{0\}$, in which case it is unique.
\end{cor}

 By tensoring with a UHF algebra and appealing to methods and results from \cite{Th2} and \cite{Th3} we obtain in Section \ref{xxx} the following corollaries. For the formulation we denote by $S^{\alpha}_\beta$ the simplex of $\beta$-KMS states for a flow $\alpha$ on a unital $C^*$-algebra, and we recall the notion of 'strong affine isomorphism': Two compact convex sets are strongly affinely isomorphic when there is an affine bijection between them which restricts to a homeomorphism between the sets of extremal points. For Bauer simplices this is the same as affine homeomorphism, but in general it is a weaker notion. See \cite{Th3}. 
 
\begin{cor}\label{27-10-20} Let $U$ be a UHF algebra and let $pA_{\{1/2\}}p$ be the algebra from Corollary \ref{08-11-20f}. Consider a compact set $K$ of real numbers and let $\mathbb I$ be a finite or countably infinite collection of intervals in $\mathbb R$ such that $I = \mathbb R$ for at least one $I \in \mathbb I$. For each $I \in \mathbb I$ choose a compact metrizable Choquet simplex $S_I$ and for $\beta \in K$ set $\mathbb I_\beta = \left\{ I \in \mathbb I: \ \beta \in I\right\}$. There is a $2\pi$-periodic flow $\theta$ on $U \otimes pA_{\{1/2\}}p$ such that
\begin{itemize}
\item for $\beta \neq 0$ there is $\beta$-KMS state for $\theta$ if and only if $\beta \in K$,
\item for each $I \in \mathbb I$ and $\beta \in (I\cap K) \backslash \{0\}$ there is a closed face $F_I$ in $S^{\theta}_\beta$ strongly affinely isomorphic to $S_I$, and
\item for each $\beta \in K \backslash \{0\}$ and each $\omega \in S^\theta_\beta$ there is a unique norm-convergent decomposition
$$
\omega = \sum_{\beta \in \mathbb I_\beta} \omega_I \ ,
$$
where $\omega_I \in F_I$. 
\end{itemize} 
\end{cor}

 \begin{cor}\label{27-10-20a} Let $U$ be a UHF algebra and let $pA_{\{1/2\}}p$ be the algebra from Corollary \ref{08-11-20f}. Let $K$ be a compact set of real numbers. There is a $2\pi$-periodic flow $\theta$ on $U \otimes pA_{\{1/2\}}p$ such that $S^\theta_\beta = \emptyset$ if and only if $\beta \notin K \cup \{0\}$, and for $\beta, \beta' \in K \cup \{0\}$ the simplexes $S^{\theta}_\beta$ and $S^\theta_{\beta'}$ are not strongly affinely isomorphic unless $\beta = \beta'$. 
 \end{cor}

The algebra $pA_{\{1/2\}}p$ is not UHF and we raise therefore the following question; a variation of Question 6.2 in \cite{Th3}.

\begin{q}\label{q} Does Corollary \ref{08-11-20f} hold when $pA_{\{1/2\}}p$ is replaced by a UHF algebra, e.g. the CAR algebra ? 
\end{q}

In relation to Question \ref{q} it is intriguing to observe that the range of the trace on the $K_0$-group for the $C^*$-algebra $pA_{\{1/2\}}p$ is the same as for the CAR algebra. But the rich supply of flows on $ pA_{\{\frac{1}{2}\}}p$ which we construct here depends very much on the presence of the large subgroup of $K_0(pA_Fp)$ which is annihilated by the trace; a subgroup sometimes called the group of infinitesimals. Still, it is beginning to seem plausible that it is due to lack of imagination that we can not answer Question \ref{q} positively.

 \section{Preparatory lemmas}

 The first lemma can be considered as an updated version of a part of the discussion in Remark 3.3 of \cite{BEH}.

 \begin{lemma}\label{26-10-20} Let $B$ be a $C^*$-algebra and $\gamma \in \Aut(B)$ an automorphism of $B$. Let $\widehat{\gamma}$ be the dual action on $B \rtimes_{\gamma} \mathbb Z$ considered as a $2 \pi$-periodic flow. For $\beta \in \mathbb R$ the restriction map $\omega \mapsto \omega|_B$ is a bijection from the $\beta$-KMS weights for $\widehat{\gamma}$ onto the traces $\tau$ on $B$ with the property that $\tau \circ \gamma = e^{-\beta} \tau$. The inverse is the map $\tau \mapsto \tau \circ P$, where $P:  B \rtimes_{\gamma} \mathbb Z \to B$ is the canonical conditional expectation.
 \end{lemma}
 \begin{proof} Let $p_1 \leq p_2 \leq \cdots$ be an approximate unit in $B$ from the Pedersen ideal $K(B)$ of $B$ with the additional property that $p_k^2 \leq p^2_{k+1}$ for all $k$, cf. \cite{Pe}. Let $u$ be the canonical unitary multiplier of $B \rtimes_\gamma \mathbb Z$ such that $ubu^* = \gamma(b)$ for $b \in B$, and let $\tau$ be a trace on $B$ with the property that $\tau \circ \gamma = e^{-\beta} \tau$. For $a,b \in B, \ n,m \in \mathbb Z$ we have that
 $$
 P(p_kau^np_k^2bu^mp_k) = P(p_kbu^m p_k^2\widehat{\gamma}_{i\beta}(au^n)p_k)  = 0
 $$
 unless $m = -n$. Set $a' = p_ka\gamma^n(p_k), \ b' = p_kb\gamma^{-n}(p_k)$. Then $a',b' \in K(B)$, and by using Proposition 5.5.2 in \cite{Pe} we find that
 \begin{align*}
 & \tau\circ P(  p_kbu^{-n}p_k\widehat{\gamma}_{i \beta} (p_kau^np_k))  = e^{-n\beta}\tau(p_kbu^{-n}p_k^2au^np_k) \\
&  =e^{-n \beta} \tau( b'u^{-n} a'u^n) =\tau(\gamma^n(b')a') \\
& =  \tau(a'\gamma^n(b')) = \tau \circ P(p_kau^n p_k^2bu^{-n}p_k) \ ,  
  \end{align*} 
showing that $\tau \circ P$ is a $\beta$-KMS functional on $p_k(B \rtimes_\gamma \mathbb Z)p_k$ for the restriction of $\widehat{\gamma}$. We shall use repeatedly that when $b \geq 0$ in $B$ we have that
\begin{equation}\label{10-11-20a}
\lim_{k \to \infty} \tau(p_kbp_k) = \tau(b) \ ,
\end{equation} 
which follows from the lower semi-continuity of $\tau$ since $\tau(p_kbp_k) = \tau(\sqrt{b}p_k^2\sqrt{b})$ and $\sqrt{b}p_k^2\sqrt{b}$ increases to $b$ as $k \to \infty$. Let $x \in D(\widehat{\gamma}_{-\frac{i\beta}{2}})$. Then
\begin{align*}
& \tau\circ P\left(\widehat{\gamma}_{- \frac{i\beta}{2}}(x) \widehat{\gamma}_{- \frac{i\beta}{2}}(x)^*\right) = \lim_{k \to \infty} \tau\circ P\left(\widehat{\gamma}_{- \frac{i\beta}{2}}(x) p_k^2 \widehat{\gamma}_{- \frac{i\beta}{2}}(x)^*\right)\\
& = \lim_{k \to \infty}\lim_{l \to \infty} \tau \left(p_lP\left(\widehat{\gamma}_{- \frac{i\beta}{2}}(x) p_k^2 \widehat{\gamma}_{- \frac{i\beta}{2}}(x)^*\right)p_l\right) \ . 
\end{align*}
We have shown above that $\tau \circ P$ is a $\beta$-KMS functional on $p_l(B \rtimes_\gamma \mathbb Z)p_l$ and when $l \geq k$ this gives 
\begin{align*}
&\tau\left(p_l P\left(\widehat{\gamma}_{- \frac{i\beta}{2}}(x) p_k^2 \widehat{\gamma}_{- \frac{i\beta}{2}}(x)^*\right)p_l\right)  = \tau\circ P\left(\widehat{\gamma}_{- \frac{i\beta}{2}}(p_lxp_k)  \widehat{\gamma}_{- \frac{i\beta}{2}}(p_lxp_k)^*\right) \\
& = \ \tau \circ P((p_lxp_k)^*p_lxp_k) = \tau \circ P(p_kx^*p_l^2xp_k)  \leq  \tau\left(p_kP(x^*x)p_k\right) \ .
\end{align*}
By using \eqref{10-11-20a} we conclude that
$$
\tau \circ P\left(\widehat{\gamma}_{- \frac{i\beta}{2}}(x) \widehat{\gamma}_{- \frac{i\beta}{2}}(x)^*\right) \ \leq \ \tau \circ P(x^*x) \ .
$$
Similarly,
\begin{align*}
& \tau \circ P(x^*x) = \lim_{k \to \infty} \tau \circ P(x^*p_k^2x) = \lim_{k \to \infty} \lim_{l \to \infty} \tau \circ P(p_lx^*p_k^2 x p_l) \ .
\end{align*}
When $l \geq k$,
\begin{align*}
&\tau \circ P(p_lx^*p_k^2 xp_l) = \tau \circ P( \widehat{\gamma}_{-\frac{i\beta}{2}}(p_k xp_l) \widehat{\gamma}_{-\frac{i\beta}{2}}(p_kxp_l)^*) \\
& = \tau \circ P(p_k \widehat{\gamma}_{-\frac{i\beta}{2}}( x)p_l^2 \widehat{\gamma}_{-\frac{i\beta}{2}}(x)^*p_k \  \leq \tau\left(p_kP( \widehat{\gamma}_{-\frac{i\beta}{2}}( x) \widehat{\gamma}_{-\frac{i\beta}{2}}(x)^*)p_k\right) ,
\end{align*}
and we find therefore that
\begin{align*}
&\tau \circ P(x^*x) \leq   \lim_{k \to \infty}  \tau \left(p_kP( \widehat{\gamma}_{-\frac{i\beta}{2}}( x) \widehat{\gamma}_{-\frac{i\beta}{2}}(x)^*)p_k\right) =  \tau \circ P( \widehat{\gamma}_{-\frac{i\beta}{2}}( x) \widehat{\gamma}_{-\frac{i\beta}{2}}(x)^*) \ .
\end{align*}
We conclude therefore first that $\tau \circ P(x^*x) = \tau \circ P( \widehat{\gamma}_{-\frac{i\beta}{2}}( x) \widehat{\gamma}_{-\frac{i\beta}{2}}(x)^*)$, and then that $\tau \circ P$ is a $\beta$-KMS weight.

Let $\omega$ be a $\beta$-KMS weight for $\widehat{\gamma}$. Then $\omega(b^*b) = \omega(bb^*)$ for all $b \in B$ because $\widehat{\gamma}_{-i\frac{\beta}{2}}(b) = b$. Using Riemann sum approximations to the integral
$$
P(a) = (2\pi)^{-1} \int_0^{2\pi} \widehat{\gamma}_t(a) \ \mathrm d t \ ,
$$
it follows from the $\widehat{\gamma}$-invariance and lower semi-continuity of $\omega$ that
\begin{equation}\label{26-10-20b}
\omega \circ P(a) \leq \omega(a)
\end{equation}
 when $ a \geq 0$ in $B\rtimes_\gamma \mathbb Z$. In particular, $\omega|_B$ is densely defined since $\omega$ is, and we conclude that $\tau = \omega|_B$ is a trace on $B$.  Let $a \geq 0$. Then 
$$
\omega \circ \gamma(a) = \omega(uau^*) = e^{-\beta}\omega\left( \widehat{\gamma}_{-\frac{i\beta}{2}}(u\sqrt{a}) \widehat{\gamma}_{-\frac{i\beta}{2}}(u\sqrt{a})^*\right) = e^{-\beta} \omega(a) \ .
$$
It follows that $\tau \circ \gamma = e^{-\beta}\tau$. It remains now only to show that $\omega = \omega \circ P$. Note that because $\omega \circ \widehat{\gamma}_t = \omega$ we find for all $b\in B, \ n \in \mathbb Z$, that
$$
\omega(p_kbu^np_k) = \begin{cases} 0, & \ n \neq 0 \\\omega(p_kbp_k) , & \ n = 0 \end{cases} \ = \ \omega(p_kP(bu^n)p_k) \ ,
$$
 implying that $\omega(p_k \ \cdot \ p_k) = \omega(p_kP( \ \cdot \ )p_k)$. Let $a \in B \rtimes_\gamma \mathbb Z$, $ a \geq 0$. Since $\omega|_B$ is a trace we can use \eqref{10-11-20a} to conclude that
$$
\lim_{k \to \infty} \omega(p_kap_k) = \lim_{k \to \infty} \omega(p_kP(a)p_k) =  \ \omega(P(a)) \ .
$$
Since $\lim_{k \to \infty} p_kap_k = a$ the lower semi-continuity of $\omega$ implies now that $\omega(a) \leq \omega(P(a))$. Combined with \eqref{26-10-20b} this yields the desired conclusion that $\omega \circ P = \omega$.

 \end{proof}

 \begin{remark}\label{11-11-20} Assume that $B$ has a trace which is unique up to multiplication by scalars and $\gamma$-invariant. It follows then from Lemma \ref{26-10-20} that the dual action will have no $\beta$-KMS weights for $\beta \neq 0$, and if in addition $B \rtimes_{\gamma} \mathbb Z$ is simple the restriction of the dual action to any corner $e(B \rtimes_\gamma \mathbb Z)e$ given by a $\widehat{\gamma}$-invariant non-zero projection $e$ will have no $\beta$-KMS states for $\beta \neq 0$ by Theorem 2.4 in \cite{Th1}. This observation applies to the flows that were shown not to be approximately inner in \cite{Ki3} and \cite{MS}.
\end{remark} 
 
Taking $\beta =0$ in Lemma \ref{26-10-20} we get a special case which we shall need. It is probably known.
 
  \begin{cor}\label{28-10-20a} Let $B$ be a $C^*$-algebra and $\gamma \in \Aut(B)$ an automorphism of $B$.  The restriction map $\omega \mapsto \omega|_B$ is a bijection from the $\widehat{\gamma}$-invariant traces $\omega$ on $B \rtimes_\gamma \mathbb Z$ onto the $\gamma$-invariant traces on $B$. The inverse is the map $\tau \mapsto \tau \circ P$, where $P:  B \rtimes_{\gamma} \mathbb Z \to B$ is the canonical conditional expectation.
\end{cor}  
  
We remark that in general $B \rtimes_\gamma \mathbb Z$ can have traces that are not invariant under the dual action $\widehat{\gamma}$ and hence are not determined by their restriction to $B$; also in cases where $B$ is UHF and $B \rtimes_\gamma \mathbb Z$ is simple. For the present purposes it is crucial that we can circumvent this issue thanks to the following lemma which is suggested by the work of Matui and Sato in \cite{MS}. In fact, most of it appears implicitly in \cite{MS}.

Recall that a $C^*$-algebra $A$ is stable when $A \otimes \mathbb K \simeq A$, where $\mathbb K$ is the $C^*$-algebra of compact operators on an infinite dimensional separable Hilbert space. 
 
\begin{lemma}\label{28-10-20} Let $B$ be a stable simple AF-algebra not isomorphic to $\mathbb K$, and let $\gamma \in \Aut(B)$ be an automorphism of $B$. There is an automorphism $\gamma' \in \Aut(B)$ such that
\begin{enumerate}
\item[a)] $\gamma'_* = \gamma_*$ on $K_0(B)$,
\item[b)] the restriction map 
$\mu \ \mapsto \ \mu|_B$
is a bijection from traces $\mu $ on $B \rtimes_{\gamma'} \mathbb Z$ onto the $\gamma'$-invariant traces on $B$, 
\item[c)] $B \rtimes_{\gamma'} \mathbb Z$ is $\mathcal Z$-stable; i.e $(B \rtimes_{\gamma'} \mathbb Z)\otimes \mathcal Z \simeq B \rtimes_{\gamma'} \mathbb Z$ where $\mathcal Z$ denotes the Jiang-Su algebra, \cite{JS}, and
\item[d)] $B \rtimes_{\gamma'} \mathbb Z$ is stable.
\end{enumerate} 
\end{lemma}
 
\begin{proof} The first step is to show that there is an isomorphism $\phi : B \to B \otimes \mathcal Z$ such that $\phi_* = {(\id_B \otimes 1_{\mathcal Z})}_*$. Since $\mathcal Z$ is strongly self-absorbing in the sense of Toms and Winter, \cite{TW}, it suffices for this to show that $B$ is $\mathcal Z$-stable. For a unital simple AF algebra this follows from Theorem A and Corollary C in \cite{CETWW}, and $\mathcal Z$-stability is invariant under stable isomorphism by Corollary 3.2 in \cite{TW}. It follows from \cite{B} that $B$ is stably isomorphic to a unital AF algebra, and we deduce in this way the existence of $\phi$.

For the second step we use \cite{Sa} to obtain an automorphism $\theta$ of $\mathcal Z$ with the weak Rohlin property; that is, with the property that for each $k \in \mathbb N$ there is a sequence $\{f_n\}$ in $\mathcal Z$ such that
\begin{itemize}
\item $0 \leq f_n \leq 1_{\mathcal Z}$ for all $n$,
\item $\lim_{n \to \infty} f_na -af_n = 0$ for all $a \in \mathcal Z$,
\item $\lim_{n \to \infty} \theta^j(f_n)f_n = 0$ for $j=1,2,3,\cdots, k$, and
\item $\lim_{n \to \infty} \tau\left(1_{\mathcal Z} - \sum_{j=0}^{k}\theta^j(f_n)\right) = 0$ ,
\end{itemize}
where $\tau$ is the trace state of $\mathcal Z$. From the first step, applied twice, it follows that there is an isomorphism $\phi : B \to B \otimes \mathcal Z \otimes \mathcal Z$ such that $\phi_* = (\id_B \otimes 1_{\mathcal Z} \otimes 1_{\mathcal Z})_*$. Since $B$ is stable there is also a $*$-isomorphism $\psi_0 : B \to \mathbb K \otimes B$ such that ${\psi_0}_* = (e \otimes \id_B)_*$, where $e$ is a minimal non-zero projection in $\mathbb K$. Set 
$$
\psi = \left(\psi_0 \otimes \id_{\mathcal Z} \otimes  \id_{\mathcal Z}\right) \circ \phi : B \to \mathbb K \otimes B \otimes \mathcal Z \otimes \mathcal Z \ 
$$
and
$$
\gamma' = \psi^{-1} \circ \left(\id_{\mathbb K} \otimes \gamma \otimes \theta \otimes \id_{\mathcal Z}\right) \circ \psi \ .
$$ 
It follows from the K\"unneth theorem that $\psi_* = (e\otimes \id_B \otimes 1_{\mathcal Z} \otimes 1_{\mathcal Z})_*$ and hence $\left(\id_{\mathbb K} \otimes \gamma \otimes \theta \otimes \id_{\mathcal Z}\right)_* \circ \psi_* = (e \otimes \gamma \otimes 1_{\mathcal Z} \otimes 1_{\mathcal Z})_*$. Thus $\gamma'_* = \gamma$.
To complete the proof it suffices to verify that b), c) and d) hold when $B$ is replaced by $\mathbb K \otimes B \otimes \mathcal Z\otimes \mathcal Z$ and $\gamma'$ by $\id_{\mathbb K} \otimes \gamma \otimes \theta \otimes \id_{\mathcal Z}$. The properties c) and d) follow immediately because $ \mathcal Z \otimes \mathcal Z \simeq \mathcal Z$, $\mathbb K \otimes \mathbb K \simeq \mathbb K$ and
$$
(\mathbb K \otimes B \otimes \mathcal Z \otimes \mathcal Z) \rtimes_{\id_{\mathbb K} \otimes\gamma \otimes \theta \otimes \id_{\mathcal Z}} \mathbb Z \ \simeq \ \mathbb K \otimes \left((B \otimes \mathcal Z)  \rtimes_{\gamma \otimes \theta} \mathbb Z \right) \otimes \mathcal Z \ .
$$
In the final step we verify b). For this, set $B' = \mathbb K \otimes B \otimes \mathcal Z$ and $\gamma'' = \id_{\mathbb K} \otimes \gamma \otimes \id_{\mathcal Z}$. Then 
$$
( \mathbb K \otimes B \otimes \mathcal Z \otimes \mathcal Z) \rtimes_{\id_{\mathbb K} \otimes \gamma \otimes \theta \otimes \id_{\mathcal Z}} \mathbb Z \ \simeq \ (B'\otimes \mathcal Z) \rtimes_{\gamma'' \otimes \theta} \mathbb Z \ ,
$$
and it suffices to verify that b) holds when $B$ is replaced by $B' \otimes \mathcal Z$ and $\gamma'$ by $\gamma''\otimes \theta$. Let $P : (B' \otimes \mathcal Z) \rtimes_{\gamma''\otimes \theta} \mathbb Z \to B' \otimes \mathcal Z$ be the canonical conditional expectation. In view of Corollary \ref{28-10-20a} what remains is to consider a trace $\mu$ on $(B '\otimes \mathcal Z) \rtimes_{\gamma''\otimes \theta} \mathbb Z$ and show that $\mu = \mu \circ P$. For this, let $\{q_n\}$ be an approximate unit in $B'$ consisting of projections. Since $q_n \otimes 1_{\mathcal Z}$ is a projection and therefore contained in the Pedersen ideal, $\mu(q_n \otimes 1_{\mathcal Z}) < \infty$, cf. \cite{Pe}. When $a \geq 0$ in $ (B' \otimes \mathcal Z) \rtimes_{\gamma''\otimes \theta} \mathbb Z $, 
$$
\mu(a) = \lim_n \mu\left(\sqrt{a}(q_n \otimes 1_{\mathcal Z}) \sqrt{a}\right) =  \lim_n \mu\left((q_n \otimes 1_{\mathcal Z})a(q_n\otimes 1_{\mathcal Z})\right) \ ,
$$
and similarly $
\mu(P(a)) =   \lim_n \mu\left((q_n \otimes 1_{\mathcal Z})P(a)(q_n\otimes 1_{\mathcal Z})\right)$.
It suffices therefore to show that 
$$
\mu\left((q_n \otimes 1_{\mathcal Z})a(q_n\otimes 1_{\mathcal Z})\right) = \mu\left((q_n \otimes 1_{\mathcal Z})P(a)(q_n\otimes 1_{\mathcal Z})\right)
$$ 
for all $n$. For this note that $x \mapsto \mu\left((q_n \otimes 1_{\mathcal Z})x(q_n\otimes 1_{\mathcal Z})\right)$ extends to a bounded positive linear functional of norm $\mu\left(q_n \otimes 1_{\mathcal Z}\right)$ on $(B' \otimes \mathcal Z) \rtimes_{\gamma''\otimes \theta} \mathbb Z$. It suffices therefore to consider positive elements $b \in B'$, $z \in \mathcal Z$ and $k \in \mathbb Z \backslash \{0\}$, and show that 
\begin{equation}\label{23-10-20}
\mu((q_n \otimes 1_{\mathcal Z})(b\otimes z) u^k(q_n\otimes 1_{\mathcal Z})) = 0
\end{equation}
when $u$ is the canonical unitary in the multiplier algebra of $(B' \otimes \mathcal Z) \rtimes_{\gamma''\otimes \theta} \mathbb Z$ such that $\Ad u = \gamma'' \otimes \theta$ on $B' \otimes \mathcal Z$. Since the complex conjugate of $\mu((q_n \otimes 1_{\mathcal Z})(b\otimes z) u^{-k}(q_n\otimes 1_{\mathcal Z}))$ is $\mu((q_n \otimes 1_{\mathcal Z})({\gamma''}^k(b)\otimes \theta^k(z)) u^{k}(q_n\otimes 1_{\mathcal Z}))$
we may assume $k \geq 1$. Let $\epsilon > 0$. Since 
$a \mapsto \mu(q_n \otimes a)$ is a bounded trace on $\mathcal Z$ it must be a scalar multiple of the unique trace state $\tau$ of $\mathcal Z$. It follows therefore from the properties of $\theta$ that there are elements $0 \leq g_j \leq 1_{\mathcal Z}, \ j = 0,1,2, \cdots , k$, in $\mathcal Z$ such that 
\begin{align}
&\left\|g_jz-zg_j\right\| \leq \epsilon \ \text{for all} \ j \ ,\label{23-10-20a} \\
& \left\|g_ig_j \right\| \leq \epsilon \ \text{for all} \ i,j, \ i \neq j , \label{23-10-20b} \\
& \left\|\theta^k(g_j)g_j\right\| \leq \epsilon \ \text{for all } \ j, \ \text{and} \label{23-10-20c} \\
&\left|\mu(q_n \otimes (1_{\mathcal Z}- \sum_{j=0}^{k} g_j)) \right| \ \leq \ \epsilon \mu(q_n\otimes 1_{\mathcal Z}). \ \label{23-10-20d}
\end{align}
In the following, when $s$ and $t$ are complex numbers such that for all $\delta > 0$ the $\epsilon$ in \eqref{23-10-20a}-\eqref{23-10-20d} can be chosen so small that $|s-t|\leq  \delta$, we will write $s \sim t$. Set
$$
y = 1_{\mathcal Z} - \sum_{j=0}^k g_j \ .
$$
It follows from \eqref{23-10-20a} that
$$
\mu((q_n \otimes 1_{\mathcal Z})(b\otimes z y) u^k(q_n\otimes 1_{\mathcal Z})) \ \sim \ \mu((q_n \otimes 1_{\mathcal Z})(b\otimes yz ) u^k(q_n\otimes 1_{\mathcal Z}))\ ,
$$
and from the Cauchy-Schwarz inequality that
\begin{equation}\label{23-10-20e}
\begin{split}
&\left|\mu((q_n \otimes 1_{\mathcal Z})(b\otimes yz ) u^k(q_n\otimes 1_{\mathcal Z}))\right| \ =  \ \left|\mu((q_n \otimes y)(b\otimes z ) u^k(q_n\otimes 1_{\mathcal Z}))\right| \\
& \leq  \ \sqrt{  \mu(q_n \otimes y^2)}\|b\|\|z\| \sqrt{\mu(q_n\otimes 1_{\mathcal Z})} \ .
\end{split}
\end{equation}
Note that it follows from \eqref{23-10-20b} that $\left\| (\sum_{j=0}^{k} g_j)^2 - \sum_{j=0}^{k} g_j^2\right\| \leq (k+1)^2\epsilon$ and hence
$$(\sum_{j=0}^{k} g_j)^2 \leq \sum_{j=0}^{k} g_j^2 + (k+1)^2 \epsilon 1_{\mathcal Z}\ \leq \  \sum_{j=0}^{k} g_j + (k+1)^2 \epsilon 1_{\mathcal Z}\ .
$$ 
It follows that $y^2 \leq y + (k+1)^2 \epsilon 1_{\mathcal Z}$ and then from \eqref{23-10-20d} that
$$
0 \ \leq  \ \mu(q_n \otimes y^2) \ \leq  \ ((k+1)^2 +1) \mu(q_n\otimes 1_{\mathcal Z}) \epsilon \ .
$$
Combined with \eqref{23-10-20e} this shows that $\mu((q_n \otimes 1_{\mathcal Z})(b\otimes yz ) u^k(q_n\otimes 1_{\mathcal Z})) \sim 0$, so to obtain \eqref{23-10-20} it suffices to show that $\mu((q_n \otimes 1_{\mathcal Z})(b\otimes zg_j) ) u^k(q_n\otimes 1_{\mathcal Z})) \sim 0$ for each $j$. It follows from \eqref{23-10-20a} that
$$
\mu((q_n \otimes 1_{\mathcal Z} )(b\otimes zg_j) ) u^k(q_n\otimes 1_{\mathcal Z})) \ \sim \ \mu((q_n \otimes 1_{\mathcal Z})(b\otimes \sqrt{g_j} z\sqrt{g_j}) ) u^k(q_n\otimes 1_{\mathcal Z})) \ . 
$$ 
Using Proposition 5.5.2 in \cite{Pe} for the second equality we get that 
\begin{align*}
&\mu((q_n \otimes 1_{\mathcal Z})(b\otimes \sqrt{g_j} z\sqrt{g_j})  u^k(q_n\otimes 1_{\mathcal Z})) \\
&= \mu((q_n \otimes \sqrt{g_j})(b\otimes  z \sqrt{g_j} ) u^k(q_n\otimes 1_{\mathcal Z})) \\
& = \mu((q_n \otimes 1_{\mathcal Z})(b\otimes  z \sqrt{g_j} ) u^k(q_n\otimes \sqrt{g_j})) \\
& = \mu((q_n \otimes 1_{\mathcal Z})(b{\gamma''}^k(q_n)\otimes  z \sqrt{g_j} \theta^k(\sqrt{g_j}) ) u^k(q_n\otimes 1_{\mathcal Z})) \ .
\end{align*}
It follows therefore from \eqref{23-10-20c} that
$ \mu((q_n \otimes 1_{\mathcal Z})(b\otimes zg_j) ) u^k(q_n\otimes 1_{\mathcal Z})) \sim 0$ as desired.
\end{proof}

The next lemma is well-known, but I could not find it in the litterature.

\begin{lemma}\label{23-10-20g} Let $B$ be an AF algebra. There is a bijective correspondence between traces $\tau$ on $B$ and the set of non-zero positive homomorphisms $ K_0(B) \to \mathbb R$. The bijection is given by the formula $\tau_*[e] = \tau(e)$ when $e$ is a projection in $B$.
\end{lemma} 
 \begin{proof} Left to the reader.
 \end{proof}

\section{Proof of the main result}

Let $F$ and $F_1$ be closed subsets of $[0,1]$ such that $\left(F \cup F_1\right) \cap \{0,1\} = \emptyset$, $F \neq \emptyset$ and either
\begin{equation}\label{08-11-20}
1/2 \notin F \cup F_1\ ,
\end{equation}
or
\begin{equation}\label{08-11-20a}
1/2 \in  F \cap F_1 \ .
\end{equation}
As in \cite{BEH} we denote by $\alpha$ the automorphism of $G_0$ given by multiplication with the function $t \mapsto \frac{t}{1-t}$. Set $G = \oplus_\mathbb Z G_0$ and
$$
G^+ = \left\{ x \in G: \ \sum_{n\in \mathbb Z}\alpha^n(x_n)(t) > 0 \ \forall t \in F,  \ \sum_{n\in \mathbb Z}x_n(t) > 0 \ \forall t \in F_1   \right\} \cup \{0\} \  .
$$
In order to work with $(G,G^+)$ we introduce the following notation. When $a,b \in G_0$, set 
$$
[[a]]_n = \begin{cases} 0, & \ n \neq 0 \\ a, & \ n = 0 \ , \end{cases}
$$
and
$$
[[-b,a,b]]_n = \begin{cases} 0, & \ n \leq -2 \\ -b , & \ n = -1, \\ a , & \ n =0, \\ b, & \ n =1, \\ 0 , & \ n \geq 2 \ . \end{cases}
$$
 Then $[[a]] = \left([[a]]_n\right)_{n \in \mathbb Z} \in G, [[-b,a,b]] = \left([[-b,a,b]]_n\right)_{n \in \mathbb Z}\in G$. In the following we denote by $C_\mathbb R(F')$ the set of continuous real-valued functions on a compact subset $F'$ of $\mathbb R$.

\begin{lemma}\label{22-10-20a} Let $F'$ be a closed subset of $[0,1]$ such that $F' \cap \{0,1\} = \emptyset$. Then
\begin{itemize}
\item $\left\{f|_{F'}: \ f \in \mathbb Z[t] \right\}$ is dense in $C_{\mathbb R}(F')$, and
\item the functions of the form $ F' \ni t \mapsto \frac{2t-1}{t(1-t)}f(t)$ for some $f \in \mathbb Z[t]$ are dense in $C_\mathbb R(F')$ when $1/2 \notin F'$, and in $\{g \in C_\mathbb R(F'): \ g(1/2) = 0 \}$ when $1/2 \in F'$.
\end{itemize}
\end{lemma}
\begin{proof} The first item follows from the argument in the proof of 2.1 in \cite{BEH} since we assume $F' \cap \{0,1\} = \emptyset$, and the second follows from the first; immediately when $1/2 \notin F'$ and when $1/2 \in F'$ by first observing that the Stone-Weierstrass theorem shows that the functions of the form $F' \ni t \mapsto (2t-1)h(t)$ for some $h \in C_\mathbb R(F')$ are dense in $\{g \in C_\mathbb R(F'): \ g(1/2) = 0 \}$.
\end{proof}

\begin{lemma}\label{20-10-20} $(G,G^+)$ is a simple dimension group. 
\end{lemma}
\begin{proof} It is straightforward to show that $G^+ \cap (-G^+) = \{0\}$ and that $G$ is unperforated. We show that 
$G$ has the Riesz interpolation property. Let $x^i =(x_n^i)_{n \in \mathbb Z}, \ y^j =(y_n^j)_{n \in \mathbb Z} \in G, \ i,j = 1,2$, and assume that $y^j - x^i \in G^+$ for all $i,j \in 1,2$. If $x^{i'} = y^{j'}$ for some $i',j'$, set $z = x^{i'}$. Then $x^i \leq z \leq y^j$ in $G$ for all $i,j$. Otherwise $\sum_{n \in \mathbb Z} \alpha^n(x_n^i)(t) < \sum_{n \in \mathbb Z} \alpha^n(y_n^j)(t) \ \forall t \in F$ and  $\sum_{n \in \mathbb Z} x_n^i(t) < \sum_{n \in \mathbb Z} y_n^j(t) \ \forall t \in F_1$ for all $i,j$. It follows from Lemma \ref{22-10-20a} that there is an element $a \in \mathbb Z[t]$ such that
$$
\sum_{n \in \mathbb Z} x_n^i(t) < a(t) < \sum_{n \in \mathbb Z} y_n^j(t) \ 
$$
for all $t \in F_1$ and all $i,j$. We can also find $h\in C_\mathbb R(F)$ such that
$$
\sum_{n \in \mathbb Z} \alpha^n(x_n^i)(t) - a(t) < h(t) < \sum_{n \in \mathbb Z} \alpha^n(y_n^j)(t) - a(t) 
$$
for all $t \in F$ and all $i,j$, and such that $h(\frac{1}{2}) = 0$ when $\frac{1}{2} \in F$, since then
$$
\sum_{n \in \mathbb Z} \alpha^n(x_n^i)(1/2) - a(1/2) = \sum_{n \in \mathbb Z} x_n^i(1/2) - a(1/2) \ <  \ 0
$$
and
$$
\sum_{n \in \mathbb Z} \alpha^n(y_n^j)(1/2) - a(1/2) = \sum_{n \in \mathbb Z} y_n^j(1/2) - a(1/2) \ > \ 0 \ .
$$
It follows from Lemma \ref{22-10-20a} that we can find $b \in \mathbb Z[t]$ such that
$$
\sum_{n \in \mathbb Z} \alpha(x_n^i)(t) - a(t) < \frac{(2t-1)}{t(1-t)}b(t) < \sum_{n \in \mathbb Z} \alpha^n(y_n^j)(t) - a(t)  
$$
for all $t \in F$ and all $i,j$. It is then straightforward to check that
$$
x^i \leq [[-b,a,b]] \leq y^j
$$
in $G$ for all $i,j$. To show that $(G,G^+)$ is a simple dimension group it remains now only to show that every non-zero element $x = (x_n)_{n \in \mathbb Z}$ of $G^+$ is an order unit. Let $y \in G$. Note that there is a $k \in \mathbb N$ such that $\sum_{n \in \mathbb Z} y_n(t) < k \sum_{n \in \mathbb Z} x_n(t)$ for all $t \in F_1$ and $\sum_{n \in \mathbb Z} \alpha^n(y_n)(t) < k \sum_{n \in \mathbb Z} \alpha^n(x_n)(t)$ for all $t \in F$. It follows that $y \leq kx$ in $G$. 
\end{proof}

It follows from Lemma \ref{20-10-20} and \cite{EHS} that there is a simple AF algebra $B_F$ whose $K_0$-group and dimension range is isomorphic to $(G ,G^+)$. Furthermore, it follows from \cite{E} that $B_F$ is stable and that there is an automorphism $\gamma$ of $B_F$ such that
\begin{equation}\label{23-10-20f}
\gamma_*\left((x_n)_{n \in \mathbb Z}\right) = \left(\alpha(x_{n+1})\right)_{n \in \mathbb Z} \ 
\end{equation}
under the identification $K_0(B_F) = G$. Using Lemma \ref{20-10-20} we choose $\gamma$ such that it has the following additional properties:
\begin{poem}\mbox{}\\[-\baselineskip]
  \begin{enumerate}\label{listref}
    \item The restriction map $\mu \ \mapsto \ \mu|_{B_F}$ is a bijection from the traces $\mu$ on $   B_F \rtimes_{\gamma} \mathbb Z$ onto the $\gamma$-invariant traces on $B_F$. \\
    \item $B_F \rtimes_{\gamma} \mathbb Z$ is stable. \\   
    \item $B_F \rtimes_{\gamma} \mathbb Z$ is $\mathcal Z$-stable.
    \end{enumerate}
\end{poem}

Set $C = B_F \rtimes_\gamma \mathbb Z$. No power of $\gamma$ is inner since $\gamma_*^n \neq \id_G$ for $n \neq 0$ and it follows therefore from \cite{Ki1} that $C$ is simple. The Pimsner-Voiculescu exact sequence shows that $K_1(C) = 0$ since $\id_G - \gamma_*$ is injective, and that
$$
K_0(C) = \coker (\id_G-\gamma_*) = G/\left(\id_G - \gamma_*\right)(G) \ .
$$
Under this identification the map $K_0(B_F) \to K_0(C)$ induced by the inclusion $B_F \subseteq C$ is the quotient map $q : G \to  G/\left(\id_G - \gamma_*\right)(G)$. Hence $G^+/\left(\id_G - \gamma_*\right)(G) \subseteq K_0(C)^+$

\begin{lemma}\label{27-10-20d} Let $(x_n)_{n \in \mathbb Z} \in G$. Then $\sum_{n\in \mathbb Z} \alpha^n(x_n) = 0$ if and only if $(x_n)_{n \in \mathbb Z} = (\id_G - \gamma_*)\left((y_n)_{n \in \mathbb Z}\right)$ for some $(y_n)_{n \in \mathbb Z} \in G$.
\end{lemma}
\begin{proof} Note that $(\id_G - \gamma_*)\left((y_n)_{n \in \mathbb Z}\right) = \left(y_n -\alpha(y_{n+1})\right)_{n \in \mathbb Z}$ and that
$$
\sum_{k\in \mathbb Z} \alpha^k\left(y_k - \alpha(y_{k+1})\right) \  = \ 0 \ .
$$
 For the converse, choose $L\in \mathbb N$ so big that $x_n =0$ when $|n| \geq L$. Set $y_n = 0$ when $n > L$ and
$$
y_L = x_L , \ y_{L-1} = x_{L-1} + \alpha(x_L), \ y_{L-2} = x_{L-2} + \alpha(x_{L-1})+\alpha^2(x_L) , \ \text{etc} .
$$ 
Then 
$$
y_{-L} = x_{-L} + \sum_{i=1}^{2L}\alpha^i(x_{-L+i}) = \sum_{i=1}^{2L}\alpha^i(x_{-L+i}) = \alpha^L\left( \sum_{i=1}^{2L}\alpha^{-L+i}(x_{-L+i})\right) = 0 \ ,
$$
and hence $y_k = 0$ when $k < -L$. It follows that $(y_n)_{n \in \mathbb Z} \in G$ and $(\id_G -\gamma_*)\left((y_n)_{n\in \mathbb Z}\right) = (x_n)_{n \in \mathbb Z}$.
\end{proof}

It follows from Lemma \ref{27-10-20d} that we can define an injective homomorphism $S :  G/\left(\id_G - \gamma_*\right)(G) \to G_0$ such that 
$$
S\left(q((x_n)_{n \in \mathbb Z})\right) \ = \ \sum_{n \in \mathbb Z} \alpha^n(x_n) \ .
$$
$S$ is surjective since $S(q([[a]])) = a$ when $a \in G_0$ and hence an isomorphism with inverse $S^{-1}$ given by
\begin{equation}\label{27-10-20h}
S^{-1}(a) = q([[a]]) \ .
\end{equation}

\begin{lemma}\label{09-10-20} Let $t \in F_1$. There is a trace $\tau_t$ on $B_F$ such that
$$
{\tau_t}_*((x_n)_{n \in \mathbb Z}) \ = \ \sum_{n \in \mathbb Z } x_n(t) \  \ \ \forall (x_n)_{n \in \mathbb Z}\in G \ ,
$$
and when $t \in F$ there is a $\gamma$-invariant trace $\tau'_t$ on $B_F$ such that 
$$
{\tau'_t}_*((x_n)_{n \in \mathbb Z}) \ = \ \sum_{n \in \mathbb Z} \left(\frac{t}{1-t}\right)^{n}x_n(t) \ \ \ \forall (x_n)_{n \in \mathbb Z}\in G \ .
$$
\end{lemma}
\begin{proof} This follows almost immediately from Lemma \ref{23-10-20g}.
\end{proof}

\begin{lemma}\label{09-10-20a} Set
$$
G^{++} = \{0\} \cup \left\{ x \in G : \ \sum_{n \in \mathbb Z} \alpha^n(x_n)(t) > 0 \ \forall t \in F\right\} \ .
$$
Then $K_0(C)^+ =  {G^{++}}/\left(\id_G - \gamma_*\right)(G)$ and 
$S$ takes $K_0(C)^+$ onto $G_F^+$.
\end{lemma}
\begin{proof} Let $x = (x_n)_{n\in \mathbb Z} \in G^{++}$. It follows from Lemma \ref{22-10-20a} that we can choose $b \in \mathbb Z[t]$ such that 
$$
\frac{1-2t}{t(1-t)}b(t) + \sum_{n \in \mathbb Z} x_n(t) \ > \ 0 \ \forall t \in F_1 \ ;
$$
albeit the arguments for this depend on whether \eqref{08-11-20} or \eqref{08-11-20a} holds. In particular, it is important here that when $F_1$ contains $1/2$ then so does $F$, because then $ \sum_{n \in \mathbb Z} x_n(1/2) = \sum_{n \in \mathbb Z} \alpha^n(x_n)(1/2) > 0$ since $x \in G^{++}$.
Set
$$
y_n = \begin{cases} x_n, & \ n \notin \{1,-1\} \\ x_1 + \alpha^{-1}(b), & \ n = 1, \\ x_{-1} - \alpha(b) , & \ n = -1 \ , \end{cases}
$$
and $y = (y_n)_{n \in \mathbb Z} \in G$. Then $\sum_{n \in \mathbb Z} \alpha^n(y_n) =  \sum_{n \in \mathbb Z} \alpha^n(x_n)$ and 
$$
\sum_{n \in \mathbb Z} y_n(t)  \ =  \ \frac{1-2t}{t(1-t)}b(t) + \sum_{n \in \mathbb Z} x_n(t) \ > \ 0 
$$ 
for all $t \in F_1$. Thus $x-y \in (\id_G - \gamma_*)(G)$ by Lemma \ref{27-10-20d} and $y \in G^+$. This shows that
$$
{G^{++}}/\left(\id_G - \gamma_*\right)(G) \subseteq {G^{+}}/\left(\id_G - \gamma_*\right)(G) \subseteq K_0(C)^+ \ .
$$
 Let $x \in K_0(C)^+\backslash \{0\} \subseteq G/\left(\id_G - \gamma_*\right)(G)$ and choose $y \in G$ such that $q(y) = x$. Set $z = [[ \sum_{n \in \mathbb Z} \alpha^n(y_n)]] \in G$ and note that $q(z) =q(y) = x$ by Lemma \ref{27-10-20d}. Let $t \in F$. The trace $\tau'_t$ from Lemma \ref{09-10-20} is $\gamma$-invariant and there is therefore a trace $\tau''_t$ on $C$ such that $\tau''_t|_{B_F} = \tau'_t$. Then ${\tau''_t}_*(x) > 0$ since $x > 0 $ and $C$ is simple. Hence
$$
0 <{\tau''_t}_*(x) = {\tau'_t}_*(z) = \ \sum_{n \in \mathbb Z} \alpha^n(y_n)(t) \ .
$$
Thus $z \in G^{++}$ and $x \in {G^{++}}/\left(\id_G - \gamma_*\right)(G)$.  We conclude that 
$K_0(C)^+ =  {G^{++}}/\left(\id_G - \gamma_*\right)(G)$ and $S(K_0(C)^+) \subseteq G_F^+$. If $a \in G_F^+$ set $x = [[a]]$. Then $x \in G^{++}$, $q(x) \in K_0(C)^+$ and $S(q(x)) = a$.
\end{proof}

\begin{lemma}\label{30-10-20} Let $\phi : G \to \mathbb R$ be a non-zero positive homomorphism and $s > 0$ a positive number such that $\phi \circ \gamma_* = s\phi$. When $s =1$ there is a unique bounded Borel measure $\mu$ on $F$ such that
\begin{equation}\label{30-10-20a}
\phi\left((x_n)_{n \in \mathbb Z}\right) \ = \ \int_{F} \sum_{n \in \mathbb Z} \alpha^n(x_n)(t) \ \mathrm{d} \mu(t) 
\end{equation}
for all $(x_n)_{n \in \mathbb Z}\in G$, and when $s \neq 1$ it follows that $\frac{s}{1+s}  \in F_1$ and there is an $r > 0$ such that
$$
\phi\left((x_n)_{n \in \mathbb Z}\right) \ = \ r \sum_{n \in \mathbb Z} x_n\left(\frac{s}{1+s}\right)  \ 
$$
for all $(x_n)_{n \in \mathbb Z}\in G$.
\end{lemma}
\begin{proof} Define $\Sigma : G \to C_{\mathbb R}(F_1) \oplus C_\mathbb R(F)$ such that 
 $$
 \Sigma((x_n)_{n \in \mathbb Z}) = \left( \sum_{n\in \mathbb Z} x_n|_{F_1}, \ \sum_{n \in \mathbb Z} \alpha^n(x_n)|_{F}\right) \ .
 $$ 
We show first that $\phi$ factorises through $\Sigma$, i.e. there is a homomorphism $\phi': \Sigma(G) \to \mathbb R$ such that $\phi = \phi' \circ \Sigma$. To this end take $x \in G$ such that $\Sigma(x) =0$. It suffices to show that $\phi(x)=0$. Let $m \in \mathbb N$. It follows from the density of $\mathbb Z[t]$ in $C_\mathbb R( F \cup F_1)$ that there is a $g_m \in G_F^+$ such that $0<g_{m}(t) < \frac{1}{m}$ for all $t \in F \cup F_1$. Then $m [[g_m]] \leq [[1]]$ in $G$ and hence $0 \leq \phi([[g_m]]) \leq \frac{1}{m} \phi([[1]])$. Since $x \ + \ [[g_m]] \in G^+$ we also have that $0 \leq \phi(x) + \phi([[g_m]]) \leq \phi(x) \ + \  \frac{1}{m}\phi([[1]])$. Similarly, since $x \ - \ [[g_m]] \in - G^+$ we have that $\phi(x) \ - \ \frac{1}{m}\phi([[1]]) \ \leq \ \phi(x - [[g_m]]) \leq 0$. Letting $ m \to \infty$ we find that $\phi(x) = 0$ as desired. 
Consider an element $x =(x_n)_{n \in \mathbb Z} \in G$ and a positive rational number $L = \frac{k}{l} > 0$, where $k,l \in \mathbb N$, such that
$$
-L <  \sum_{n\in \mathbb Z} x_n(t) < L \ \forall t \in F_1
$$
and
$$
-L <\sum_{n \in \mathbb Z} \alpha^n(x_n)(t)  < L \ \forall t \in F \ .
$$
Then $-k [[1]] \leq l x \leq k[[1]]$ in $G$ and hence
$$
-L\phi([[1]]) \ \leq \ \phi'\left( \sum_{n\in \mathbb Z} x_n|_{F_1}, \ \sum_{n \in \mathbb Z} \alpha^n(x_n)|_{F}\right) \ \leq \ L\phi([[1]])\  .
$$
This shows that $\phi'$ is Lipschitz continuous for the supremum norm on $\Sigma(G)$. Set 
$$
X  = \left\{ (f,g) \in C_\mathbb R(F_1)\oplus C_\mathbb R(F) : f(1/2) = g(1/2) \right\} 
$$
when $1/2 \in F_1 \cap F$, in which case $\Sigma(G) \subseteq X$. Since
 $$
 \left\{(a,a+\frac{2t-1}{t(1-t)}b) : \ a,b \in G_0\right\} \ = \ \left\{ \Sigma([[-b,a,b]]): \ a,b \in G_0 \right\}
 $$
 is a dense additive subgroup of $C_\mathbb R(F_1) \oplus C_\mathbb R(F)$ when $1/2 \notin F_1 \cup F$ and of $X$ when $1/2 \in F_1 \cap F$, it follows that $\phi'$ extends by continuity to a linear map on $C_\mathbb R(F_1) \oplus C_\mathbb R(F)$ in the first case, and on $X$ in the second. Assume first that $1/2 \in F_1 \cap F$.  Applying the Riesz representation theorem to the positive functional 
$$
f \ \mapsto \ \phi'(f,f(1/2))
$$
on $C_\mathbb R(F_1)$ and the positive functional
$$
g \ \mapsto \ \phi'(0,g)
$$
on $\{g \in C_\mathbb R(F) : \ g(1/2) = 0\}$ we obtain bounded Borel measures $\mu_1$ on $F_1$ and $\mu_2$ on $F \backslash \{1/2\}$ such that 
$$
 \phi'(f,g) = \int_{F_1} f \ \mathrm{d}\mu_1 + \int_{F \backslash \{1/2\}} g - g(1/2) \ \mathrm d\mu_2 \ 
 $$
 for all $(f,g) \in X$. Note that $\Sigma \circ \gamma_* = \left(\alpha \times \id\right) \circ \Sigma$. Since $\phi \circ \gamma_* = s\phi_*$ it follows that
\begin{equation}\label{30-10-20c}
 \begin{split}
& \int_{F_1} \frac{t}{1-t}f(t) \ \mathrm{d}\mu_1(t) + \int_{F \backslash \{1/2\}} g(t) - g(1/2) \ \mathrm d\mu_2(t) \\
& \ \ \ \ \ \ \ \ \ \ = \ s\int_{F_1} f(t) \ \mathrm{d}\mu_1(t) + s\int_{F \backslash \{1/2\}} g(t) - g(1/2) \ \mathrm d\mu_2(t) 
\end{split} 
 \end{equation}
 for all $(f,g) \in X$. If $s=1$ it follows that $\mu_1$ is concentrated at $\frac{1}{2}$. Furthermore, since
 $$
g(1/2) \left(\mu_1(\{1/2\})    -  \mu_2(F \backslash \{1/2\})\right) + \int_{F \backslash \{1/2\}}g \ \mathrm d\mu_2  = \phi'(g(1/2), g) \geq 0
$$
when $g \geq 0$ in $C_\mathbb R(F)$, it follows that $\mu_1(\{1/2\}) \geq \mu_2(F \backslash \{1/2\})$. We can therefore define $\mu$ on $F$ such that 
$$
\mu(B) = \mu_2(B \backslash \{1/2\}) + \left(\mu_1(\{1/2\}) \ - \ \mu_2(F \backslash \{1/2\})\right) \mu_1(B \cap \{1/2\})
$$
 for every Borel set $B \subseteq F$. Then \eqref{30-10-20a} holds. Uniqueness of $\mu$ follows from the density of
 \begin{equation}\label{10-11-20f}
 \left\{ \sum_{n \in \mathbb Z} \alpha^n(x_n)|_F : \ (x_n)_{n \in \mathbb Z} \in G \right\}
\end{equation}
 in $C_\mathbb R(F)$.

 If instead $s \neq 1$ it follows from \eqref{30-10-20c} that $\mu_1$ is concentrated at a point $t' \in F_1$ such that $s = \frac{t'}{1-t'}$ while $\mu_2 = 0$. Since $\frac{s}{1+s} = t'$, this completes the proof when $\frac{1}{2} \in F_1 \cap F$. The case $\frac{1}{2} \notin (F_1 \cup F)$ is similar and slightly simpler; we leave it to the reader.  
   
\end{proof}

\begin{lemma}\label{05-11-20c} The cone of traces on $B_F \rtimes_\gamma \mathbb Z$ is in affine bijection with the cone of bounded Borel measures on $F$. The trace $\tau_m$ defined by a bounded Borel measure $m$ is given by the formula
\begin{equation}\label{09-11-20}
\tau_m (a) \ = \ \int_{F} \tau'_t(P(a)) \ \mathrm{d} m(t) \ ,
\end{equation}
where $\tau'_t$ is the $\gamma$-invariant trace on $B_F$ from Lemma \ref{09-10-20} and $P : B_F \rtimes_\gamma \mathbb Z \to B_F$ is the canonical conditional expectation.
\end{lemma}
\begin{proof} It follows from Corollary \ref{28-10-20a} that \eqref{09-11-20} defines a ($\widehat{\gamma}$-invariant) trace on $B_F \rtimes_\gamma \mathbb Z$. If $\tau$ is a trace on $B_F \rtimes_\gamma \mathbb Z$ the homomorphism $(\tau|_{B_F})_* : G \to \mathbb R$ is positive and $\gamma_*$-invariant and it follows from the first part of Lemma \ref{30-10-20} and Lemma \ref{09-10-20} that there is a bounded Borel measure $m$ on $F$ such that
$$
(\tau|_{B_F})_*[f] = \int_F \tau'_t(f) \ \mathrm{d} m(t) \ 
$$
for all projections $f \in B_F$. It follows then from Lemma \ref{23-10-20g} first that $\tau|_{B_F} = \tau_m|_{B_F}$ and then thanks to first of the Additional properties \ref{listref} that $\tau = \tau_m$. Hence the map $m \mapsto \tau_m$ is surjective. It is injective because the functions $t \mapsto \tau'_t(a), \ a \in B_F\rtimes_\gamma \mathbb Z$, are dense in $C_\mathbb R(F)$. Indeed, this collection of functions contains all the functions \eqref{10-11-20f}.
\end{proof}

The $0$-KMS weights for $\widehat{\gamma}$ are by definition the $\widehat{\gamma}$-invariant traces and by the first of the Additional properties \ref{listref} all traces on $B_F \rtimes_\gamma \mathbb Z$ are $\widehat{\gamma}$-invariant. Hence Lemma \ref{05-11-20c} gives also a description of all $0$-KMS weights for $\widehat{\gamma}$.

\begin{lemma}\label{05-11-20d} Let $\widehat{\gamma}$ be the dual action on $B_F \rtimes_\gamma \mathbb Z$ considered as a $2 \pi$-periodic flow. For $\beta \in \mathbb R\backslash \{0\}$ there is a $\beta$-KMS weight for $\widehat{\gamma}$ if and only if $\frac{e^{-\beta}}{1+e^{-\beta}} \in F_1$, in which case it is unique up to multiplication by scalars.
\end{lemma} 
\begin{proof} If $\frac{e^{-\beta}}{1+e^{-\beta}} \in F_1$ the trace $\tau_t$ on $B_F$ from Lemma \ref{09-10-20} with $t = \frac{e^{-\beta}}{1+e^{-\beta}}$ will have property that $\tau_t \circ \gamma = \frac{t}{1-t} \tau_t$. This follows from the formula defining $\tau_t$ and Lemma \ref{23-10-20g}. Since $\frac{t}{1-t} = e^{-\beta}$ it follows from Lemma \ref{26-10-20} that $\tau_t \circ P$ is a $\beta$-KMS weight for $\widehat{\gamma}$. If $\omega$ is a another $\beta$-KMS weight for $\widehat{\gamma}$ it follows from Lemma \ref{26-10-20} that $\omega|_{B_F}$ is a trace on $B_F$ such that $\omega|_{B_F} \circ \gamma = e^{-\beta} \omega|_{B_F}$. The second part of Lemma \ref{30-10-20} implies that ${\omega|_{B_F}}_* = r {\tau_t}_*$ for some $r > 0$ and then Lemma \ref{23-10-20g} and Lemma \ref{26-10-20} in combination show that $\omega = r \tau_t \circ P$, i.e. $\tau_t\circ P$ is the unique $\beta$-KMS weight up to multiplication by scalars. Conversely, if $\omega$ is a $\beta$-KMS weight for $\widehat{\gamma}$ the homomorphism ${\omega|_{B_F}}_* : G \to \mathbb R$ will be non-zero and have the property that ${\omega|_{B_F}}_* \circ \gamma_* = e^{-\beta}{\omega|_{B_F}}_*$. By Lemma \ref{30-10-20} this implies that $\frac{e^{-\beta}}{1+e^{-\beta}} \in F_1$ because we assume that $\beta\neq 0$.
\end{proof}

Let $e \in B_F$ be a projection such that $[e] = [[1]]$ in $G$.

\begin{lemma}\label{05-11-20a} Let $\widehat{\gamma}$ be the dual action on $ B_F \rtimes_\gamma \mathbb Z$ considered as a $2 \pi$-periodic flow. 
\begin{itemize} 
\item All trace states of $e(B_F \rtimes_\gamma \mathbb Z)e$ are invariant under the restriction of the dual action $\widehat{\gamma}$ and the tracial state space $T(e( B_F \rtimes_\gamma \mathbb Z)e)$ is affinely homeomorphic to the Bauer simplex of Borel probability measures on $F$.
\item For $\beta \neq 0$ there is a $\beta$-KMS state for the restriction of $\widehat{\gamma}$ to $e( B_F \rtimes_\gamma \mathbb Z)e$ if and only if $\frac{e^{-\beta}}{1+e^{-\beta}} \in F_1$, in which case it is unique.
\end{itemize}
\end{lemma}
\begin{proof} The first statement follows from Proposition 4.7 in \cite{CP} combined with Lemma \ref{05-11-20c} above, and the second statement follows from Theorem 2.4 in \cite{Th2} and Lemma \ref{05-11-20d} above.
\end{proof}

 Since $C$ is simple, $eCe$ is stably isomorphic to $C$ by Browns result, \cite{B}, and hence $\left(K_0(eCe),K_0(eCe)^+\right) = \left(K_0(C),K_0(C)^+\right)$. Let $p \in A_F$ be a projection representing the constant function $1 \in G_0$, i.e. $[p] = 1$ in $G_0$. It follows from Lemma \ref{09-10-20a} that $S$ induces an isomorphism 
\begin{equation}\label{28-10-20g}
\left(K_0(eCe),K_0(eCe)^+,[e]\right) \simeq (K_0(pA_Fp),K_0(pA_Fp)^+, [p]) \ 
\end{equation}
of partially ordered groups with order unit. Let $T(eCe)$ and $T(pA_Fp)$ denote the tracial state spaces of $eCe$ and $pA_Fp$, respectively. 

\begin{lemma}\label{27-10-20i} There is an affine homeomorphism $\lambda : T(eCe) \to T(pA_Fp)$ such that 
\begin{equation}\label{10-11-20}
\lambda(\tau)_*(S(x)) = \tau_*(x)
\end{equation} 
for all $\tau \in T(eCe)$ and $x \in K_0(eCe)$.
 \end{lemma}
 \begin{proof} By Proposition 4.7 in \cite{CP} we can identify $T(eCe)$ with the set of traces $\tau$ on $C$ with the property that $\tau(e) =1$, and similarly, $T(pA_Fp)$ with the set of traces $\tau'$ on $A_F$ with the property that $\tau'(p) =1$. Let $\Hom^+(G,\mathbb R)$ denote the set of positive homomorphisms $G \to \mathbb R$ and, similarly, $\Hom^+(G_0,\mathbb R)$ the set of positive homomorphisms $G_0 \to \mathbb R$. Thanks to the first of the Additional properties \ref{listref} combined with Lemma \ref{23-10-20g} we can identify, in the natural way, $T(eCe)$ with
 $$
 X_1 =\left\{ \phi \in \Hom^+(G,\mathbb R) : \ \phi \circ \gamma_* = \phi, \ \phi([[1]]) =1 \right\} \ ,
 $$
 and similarly, the set $T(pA_Fp)$ with the set
 $$
 X_2 = \left\{ \psi \in \Hom^+(G_0,\mathbb R) : \  \ \psi(1) =1 \right\} \ .
 $$
Using \eqref{27-10-20h} we see that we must produce an affine homeomorphism $\lambda : X_1 \to X_2$ such that
\begin{equation}\label{28-10-20e}
\lambda(\phi)(a) = \phi([[a]]) \ \ \ \forall a \in G_0 \ \forall \phi \in X_1 \ .
\end{equation}
To this end let $\phi \in X_1$. By Lemma \ref{30-10-20} there is a unique Borel probability  measure $\mu$ on $F$ such that \eqref{30-10-20a} holds. We define $\lambda(\phi) \in X_2$ such that
$$
\lambda(\phi)(a) = \int_{F} a(t) \ \mathrm{d} \mu \ 
$$
for $a \in G_0$. To construct an inverse to $\lambda$ note that every element $\psi$ of $X_2$ is given by integration with respect to a unique  Borel probability measure $\nu$ concentrated on $F$. We set
$$
\lambda^{-1}(\psi) \left((x_n)_{n \in \mathbb Z} \right) \ = \ \int_{F} \sum_{n \in \mathbb Z} \alpha^n(x_n)(t) \ \mathrm d\nu(t) \ 
$$
when $(x_n)_{n \in \mathbb Z}\in G$. Then $\lambda^{-1}(\psi) \in X_1$ and it is easy to see that $\lambda$ is a homeomorphism with inverse $\lambda^{-1}$. This completes the proof because \eqref{28-10-20e} clearly holds.

 \end{proof}

\begin{thm}\label{28-10-20f} $pA_Fp$ is $*$-isomorphic to $e( B_F \rtimes_\gamma \mathbb Z)e$.
\end{thm}
\begin{proof} The equality \eqref{10-11-20} shows that the homeomorphism $\lambda$ of Lemma \ref{27-10-20i} and the isomorphism \eqref{28-10-20g} combine to give an isomorphism of the Elliott invariants of $eCe$ and $pA_Fp$, and hence it remains only to show that the algebras $eCe$ and $pA_Fp$ both belong to a class of $C^*$-algebras for which the Elliott invariant is known to be complete. For this we appeal to Corollary D of \cite{CETWW} and we must therefore check that both algebras are separable, unital, nuclear, $\mathcal Z$-stable and satisfy the UCT. Only $\mathcal Z$-stability is not well-known. For $pA_Fp$ this follows from Theorem A and Corollary C in \cite{CETWW} and for $eCe$ it follows from the third of the Additional properties \ref{listref} combined with Corollary 3.2 in \cite{TW}. 
\end{proof}

\begin{cor}\label{08-11-20b} $ B_F \rtimes_\gamma \mathbb Z$ is $*$-isomorphic to $A_F$.
\end{cor}
\begin{proof} Both $C^*$-algebras are stable; $B_F$ because its dimension scale is the entire positive semi-group and $C$ by the second of the Additional properties \ref{listref}. Both algebras are also simple and the conclusion follows therefore from Theorem \ref{28-10-20f} and \cite{B}.
\end{proof}

By combining Theorem \ref{28-10-20f} with Lemma \ref{05-11-20a} we get the following

\begin{cor}\label{05-11-20g} $pA_Fp$ is a simple unital AF algebra with a tracial state space affinely homeomorphic to the Bauer simplex of Borel probability measures on $F$, and there is a $2\pi$-periodic flow $\theta$ on $pA_Fp$ such that for $\beta \neq 0$ there is a $\beta$-KMS state for $\theta$ if and only if $\frac{e^{-\beta}}{1+e^{-\beta}} \in F_1$, in which case it is unique. All trace states of $pA_Fp$ are $\theta$-invariant and hence the set of $0$-KMS state for $\theta$ is affinely homeomorphic to the Bauer simplex of Borel probability measures on $F$.
\end{cor}

Theorem \ref{08-11-20e} follows from Corollary \ref{05-11-20g}. To obtain Theorem \ref{08-11-20d}, note that $A_F \simeq pA_Fp \otimes \mathbb K$ by Browns result, \cite{B}, since $A_F$ is stable and simple. If $\theta$ is a flow on $pA_Fp$ with the property specified in Theorem \ref{08-11-20e}, the flow on $A_F$ which is conjugate to $\theta \otimes \id_{\mathbb K}$ under the isomorphism $A_F \simeq pA_Fp \otimes \mathbb K$ will have the property specified in Theorem \ref{08-11-20d}.

\section{Proof of the Corollary \ref{27-10-20} and Corollary \ref{27-10-20a}}\label{xxx}

We will combine Corollary \ref{08-11-20f} with the methods used in Appendix 12 of \cite{Th2} and Section 6 of \cite{Th3}. Given the collection $\mathbb I$ of intervals and simplexes $S_I, I \in \mathbb I$, in Corollary \ref{27-10-20} we get from Corollary 5.7 in \cite{Th3} a generalized gauge action $\alpha$ on $U$ with the following properties:
\begin{itemize}
\item For each $I \in \mathbb I$ and each $\beta \in I\backslash \{0\}$ there is a closed face $F_I$ in $S^{\alpha}_{\beta}$ which is strongly affinely isomorphic to $S_I$. 
\item For each $\beta \neq 0$ and each $\beta$-KMS state $\omega \in S^{\alpha}_{\beta}$ there is a unique norm-convergent decomposition
$$
\omega = \sum_{I \in \mathbb I_{\beta}} \omega_I \ 
$$
where $\omega_I \in \mathbb R^+F_I$. 
\end{itemize}
That $\alpha$ is a generalized gauge action means that there is a Bratteli diagram $\Br$ and a map $F : \Br_{Ar} \to \mathbb R$ defined on the set $\Br_{Ar}$ of arrows in $\Br$ which define a flow on $U$ in the following way. Let $\mathcal P_n$ denote the set of paths of length $n$ in $\Br$ starting at the top vertex. Extend the map $F$ to $\mathcal P_n$ such that
$$
F(\mu) = \sum_{i=1}^n F(a_i) \ ,
$$
when $\mu = a_1a_2\cdots a_n$ is made up of the arrows $a_i$. The Bratteli diagram $\Br$ is chosen such that  there are finite dimensional $C^*$-subalgebras $\mathbb F_n$ in $U$ spanned by a set of matrix units $E^n_{\mu,\mu'}, \ \mu,\mu' \in \mathcal P_n$, such that $\mathbb F_n \subseteq \mathbb F_{n+1}$ for all $n$ and
$$
U = \overline{\bigcup_n \mathbb F_n} \ . \
$$
The flow $\alpha$ is defined such that
$$
\alpha_t\left( E^n_{\mu,\mu'}\right) = e^{i (F(\mu)-F(\mu'))t} E^n_{\mu,\mu'} \ .
$$
See \cite{Th3}. In order to ensure that $\alpha$ is $2\pi$-periodic it is necessary to arrange that $F$ only takes integer values. That this is possible follows by inspection of the proof in \cite{Th3}; in fact, the only step in the proof where this is not automatic is in the proof of Lemma 5.3 in \cite{Th3}, where some real numbers $t_k$ are chosen. The only crucial property of these numbers is that they must be sufficiently big and they may therefore be chosen to be natural numbers. The resulting potential $F$ will then be integer-valued. Let $\theta^K$ be the flow from Corollary \ref{08-11-20f} and set
$$
\theta'_t = \alpha_t \otimes \theta^K_t \ ,
$$
which we will argue defines a flow $\theta'$ on $U \otimes pA_{\{1/2\}}p$ with the properties stated in Corollary \ref{27-10-20}. Since a $\beta$-KMS state for $\theta'$ will restrict to a $\beta$-KMS state for $\theta^K$ on the tensor factor $ pA_{\{1/2\}}p$, it follows from  Corollary \ref{08-11-20f} that there are no $\beta$-KMS states for $\theta'$ unless $\beta \in K \cup \{0\}$. Let $\beta \in K \backslash \{0\}$ and note that there is a unique $\beta$-KMS state $\omega_\beta$ for $\theta^K$. It remains only to show that the map $\omega \ \mapsto \ \omega \otimes \omega_\beta$ is an affine homeomorphism from $S^\alpha_\beta$ onto $S^{\theta'}_\beta$. Note that only the surjectivity is not obvious, and consider therefore a $\beta$-KMS state $\psi \in S^{\theta'}_\beta$. Let $x,y \in U, \ b \in  
pA_{\{1/2\}}p$ such that $x$ is $\alpha$-analytic. Then $x \otimes p$ is $\theta'$-analytic and $\theta'_{i \beta}(x \otimes p) = \alpha_{i\beta} (x) \otimes p$. Hence
\begin{align*}
&\psi(xy \otimes b) = \psi((x \otimes p)(y \otimes b)) = \psi((y\otimes b) \theta'_{i \beta}(x\otimes p)) \\
& = \psi((y\otimes b) (\alpha_{i \beta}(x) \otimes p)) = \psi(y\alpha_{i \beta}(x) \otimes b) \ .
\end{align*}
Thus, if $b \geq 0$ and $b \neq 0$, the map $U \ni x \mapsto \psi(x \otimes b)$ is a $\beta$-KMS functional for $\alpha$, and hence
\begin{align*}
&\psi(E^n_{\mu, \mu'} \otimes b) = \psi( \alpha_{i\beta}(E^n_{\mu, \mu'})\otimes b)) =  e^{-\beta (F(\mu) - F(\mu'))} \psi(E^n_{\mu, \mu'} \otimes b) \ .
\end{align*}
Since $\beta \neq 0$ it follows that $\psi(E^n_{\mu,\mu'} \otimes b) \ = \  0$ when $F(\mu) \neq F( \mu')$, and hence that $\psi$ factorises through the map $Q \otimes \id_{pA_{\{1/2\}}p}$, where $Q : U \to U^{\alpha}$ is the conditional expectation onto the fixed point algebra $U^\alpha$ of $\alpha$. Let $a \geq 0$ in $U$. Then
$$
pA_{\{1/2\}}p \ni b \ \mapsto \ \psi(a \otimes b) = \psi(Q(a) \otimes b)
$$
is a $\beta$-KMS functional for $\theta^K$ since $\psi$ is a $\beta$-KMS state for $\theta'$ and $Q(a)$ is fixed by $\alpha$. The uniqueness of $\omega_\beta$ implies therefore that
$$
\psi(a \otimes b) = \lambda(a)\omega_{\beta}(b)
$$
for some $\lambda(a) \in \mathbb R$ and all $b \in pA_{\{1/2\}}p$. It is then straightforward to show that $\lambda(a)$ can be defined for all $a \in U$, resulting in a $\beta$-KMS state $\lambda$ for $\alpha$. This completes the proof of Corollary \ref{27-10-20}. 

 Corollary \ref{27-10-20a} follows by choosing the family $\mathbb I$ of intervals in Corollary \ref{27-10-20} in an appropriate way, e.g. as the collection of all bounded intervals with rational endpoint, and also the Choquet simplexes $S_I$ in an appropriate way. See Section 6 in \cite{Th3}.


\begin{thebibliography}{WWWWW} 






\bibitem[BEH]{BEH} O. Bratteli, G. Elliott and R.H. Herman, {\em On the possible temperatures of a dynamical system}, Comm. Math. Phys. {\bf 74} (1980), 281-295.




\bibitem[BR]{BR} O. Bratteli and D.W. Robinson, {\em Operator Algebras
    and Quantum Statistical Mechanics I + II}, Texts and Monographs in
  Physics, Springer Verlag, New York, Heidelberg, Berlin, 1979 and 1981.
  
  
\bibitem[B]{B} L. Brown, {\em Stable isomorphism of hereditary subalgebras of $C^*$-algebras}, Pacific J. Math. {\bf 71} (1977), 335-348.  
  
 
    
     
\bibitem[CETWW]{CETWW} J. Castillejos, S. Evington, A. Tikuisis, 
S. White and W. Winter, {\em Nuclear dimension of simple $C^*$-algebras}, arXiv:1901.05853v3, Invent. Math., to appear. 

\bibitem[C]{C} F. Combes, {\em Poids associ\'e \`a une alg\`ebre
    hilbertienne \`a gauche}, Compos. Math. {\bf 23} (1971), 49--77.


\bibitem[CP]{CP} J. Cuntz and G. K. Pedersen, {\em Equivalence and traces on $C^*$-algebras}, J. Func. Analysis {\bf 33} (1979), 135-164.

\bibitem[EHS]{EHS} E. Effros, D. Handelman and C. L. Shen, {\em Dimension groups and their affine representations}, Amer. J. Math. {\bf 102} (1980), 385-407.


\bibitem[E]{E} G. Elliott, {\em On the classification of inductive limits of sequences of semisimple finite-dimensional algebras}, J. Algebra {\bf 38} (1976), 29-44.


\bibitem[JS]{JS} X. Jiang and H. Su, {\em On a simple unital projectionless $C^*$-algebra}, Amer. J. Math. {\bf 121} (2000), 359-413.

\bibitem[Ki1]{Ki1} A. Kishimoto, {\em Outer Automorphisms and Reduced Crossed Products
of Simple $C^*$-Algebras}, Comm. Math. Phys. {\bf 81} (1981), 429--435.


 \bibitem[Ki2]{Ki2} A. Kishimoto, {\em Locally representable
    one-parameter automorphism groups of AF algebras and KMS states},
  Rep. Math. Phys. {\bf 45} (2000), 333-356.


\bibitem[Ki3]{Ki3} A. Kishimoto, {\em Non-commutative shifts and crossed products}, J. Func. Analysis {\bf 200} (2003), 281-300.




\bibitem[Ku]{Ku} J. Kustermans, {\em KMS weights on $C^*$-algebras};  arXiv: 9704008v1.
  
 \bibitem[KV]{KV} J. Kustermans and S. Vaes, {\em Locally compact
    quantum groups}, Ann. Scient. \'Ec. Norm. Sup. {\bf 33}, 2000, 837-934. 
  
  
 


 
 
    
    
    
    
  

\bibitem[MS]{MS} H. Matui and Y. Sato, {\em Decomposition rank of UHF-absorbing $C^*$-algebras}, Duke Math. J. {\bf 163} (2014), 2687-2708.   
  
  
\bibitem[Pe]{Pe} G. K. Pedersen, {\em $C^*$-algebras and their automorphism groups}, Academic Press, London, 1979.


  
\bibitem[PS]{PS} R.T. Powers and S. Sakai, {\em Existence of ground
    states and KMS states for approximately inner dynamics},
  Comm. Math. Phys. {\bf 39} (1975), 273-288.







\bibitem[Sa]{Sa} Y. Sato, {\em The Rohlin property for automorphisms of the Jiang--Su
algebra}, J. Func. Analysis {\bf 259} (2010) 453--476.



 \bibitem[Th1]{Th1} K. Thomsen, {\em KMS weights on graph $C^*$-algebras}, Adv. Math. {\bf 309} (2017) 334--391.
 
\bibitem[Th2]{Th2} K. Thomsen, {\em KMS weights, conformal measures and ends in digraphs}, Adv. Oper. Th. {\bf 5} (2020), 489-607.

\bibitem[Th3]{Th3} K. Thomsen, {\em Phase transition in the CAR algebra}, Adv. Math. {\bf 372} (2020); arXiv:1612.04716v5.


\bibitem[TW]{TW} A. Toms and W. Winter, {\em Strongly self-absorbing $C^*$-algebras}, Trans. Amer. Math. Soc. {\bf 358} (2007), 3999-4029.

\end{thebibliography}
\end{document}